\documentclass[12pt]{amsart}

\usepackage{amsmath}
\usepackage{amssymb}
\usepackage{amsthm}
\usepackage{amscd}

\newtheorem{thm}{Theorem}[section]
\newtheorem{lem}[thm]{Lemma}
\newtheorem{prop}[thm]{Proposition}
\newtheorem{cor}[thm]{Corollary}
\newtheorem*{propA}{Proposition A}
\newtheorem*{lemB}{Lemma B}

\theoremstyle{definition}
\newtheorem{definition}[thm]{Definition}
\newtheorem{example}[thm]{Example}
\newtheorem*{question}{Question}

\theoremstyle{remark}
\newtheorem{rem}[thm]{Remark}
\newtheorem*{acknowledgment}{Acknowledgment}

\newcommand\Gal{\mathrm{Gal}}
\newcommand\rankzp{\mathop{\mathrm{rank}_{\mathbb{Z}_p}}}
\newcommand\rankztwo{\mathop{\mathrm{rank}_{\mathbb{Z}_2}}}

\begin{document}

\title{Class number behavior in a two-tiered tower of $\mathbb{Z}_p$-extensions}
\author{Tsuyoshi Itoh}

\maketitle

\begin{abstract}
Let $k_\infty$ be the cyclotomic $\mathbb{Z}_p$-extension field of an 
algebraic number field $k$.
Moreover, we take a $\mathbb{Z}_p$-extension $K_\infty$ over $k_\infty$.
In this paper, we study the behavior of the $p$-part of the class number of 
certain intermediate fields of $K_\infty /k$.
We also consider the structure of the unramified Iwasawa module of 
$K_\infty / k_\infty$ for several cases.
\end{abstract}

\section{Introduction}\label{introduction}

Throughout this paper, $p$ denotes a prime number.
For an algebraic extension field $\mathbb{K}$ of $\mathbb{Q}$,
we denote by $L (\mathbb{K})$ the maximal unramified abelian pro-$p$ extension
field of $\mathbb{K}$,
and put $X (\mathbb{K}) = \Gal (L (\mathbb{K})/ \mathbb{K})$.
When $\mathbb{K}$ is an algebraic number field (i.e., a finite extension field of
$\mathbb{Q}$), $X (\mathbb{K})$ is isomorphic to
the Sylow $p$-subgroup of the ideal class group of $\mathbb{K}$.

\subsection{Iwasawa's class number formula}
We shall recall Iwasawa's class number formula for a $\mathbb{Z}_p$-extension.
Let $\mathcal{K}$ be a $\mathbb{Z}_p$-extension field of an algebraic number field $k$.
For a non-negative integer $n$, we denote by $\mathcal{K}_n$ its $n$th layer
(the unique intermediate field satisfying $[ \mathcal{K}_n : k ]= p^n$).
Then there are the invariants $\lambda (\mathcal{K}/k)$, $\mu (\mathcal{K}/k)$, $\nu (\mathcal{K}/k)$
(depend only on $\mathcal{K}/k$)
such that
\[ | X (\mathcal{K}_n) | =
p^{\lambda (\mathcal{K}/k) n + \mu (\mathcal{K}/k) p^n + \nu (\mathcal{K}/k)} \]
holds for all sufficiently large $n$
(see, e.g., \cite[Chapter 13]{Was}).
These invariants $\lambda (\mathcal{K}/k)$, $\mu (\mathcal{K}/k)$, $\nu (\mathcal{K}/k)$ are
called the Iwasawa invariants of $\mathcal{K}/k$.

\subsection{Class numbers of intermediate fields of tamely ramified cyclic extensions}
There are several studies on the behavior of the $p$-part of class numbers in a 
tamely ramified cyclic extensions of $p$-power degree.
See, e.g., Ichimura \cite{Ichi}, Fukuda \cite{Fuku}.
Similar studies for more general cases (cyclic extensions which are not contained in a 
$\mathbb{Z}_p$-extension) also exist. 
See, e.g., Ichimura and Sumida-Takahashi \cite{IchiTaka} 
and subsequent papers by them.
It is a remarkable fact that the $p$-part of class number of intermediate fields behaves like 
Iwasawa's class number formula for many cyclic extensions 
not contained in a $\mathbb{Z}_p$-extension. 

Their studies treat extensions over number fields with relatively low degree.
Inspired by their studies, we propose the following:

\begin{question}
Let $F/k$ be a tamely ramified cyclic extension of degree $p^m$. 
Denote by $k_n$ the $n$th layer of the cyclotomic $\mathbb{Z}_p$-extension of $k$.
For a fixed sufficiently large integer $n$,
how the $p$-part of the class number of intermediate fields of
$F k_n /k_n$ behaves like?
\end{question}

Of course, it is not essential to restrict $F/k$ being tamely ramified.
However, we shall focus on the tamely ramified case in the latter part of this paper.

In this paper, we will give a partial answer to this question under a special situation 
that $F$ is an intermediate field of a ``two-tiered tower'' of $\mathbb{Z}_p$-extensions 
satisfying several conditions.
In particular, we also propose a method to construct a tamely ramified cyclic extension 
with arbitrary large degree whose class numbers of intermediate fields behave like 
Iwasawa's class number formula.

\subsection{Our setting and results}
Let $k$ be an algebraic number field.
In the following of this paper, we always denote by $k_\infty$ the cyclotomic
$\mathbb{Z}_p$-extension field of $k$.
(However, we will use $\mathbb{B}_\infty$ instead of $\mathbb{Q}_\infty$.)
We also denote by $k_n$ the $n$th layer of $k_\infty /k$.

Concerning the above question, we consider a $\mathbb{Z}_p$-extension field 
$K_\infty$ over $k_\infty$.
That is, $K_\infty /k_\infty /k$ is a ``two-tiered tower'' of 
$\mathbb{Z}_p$-extensions.
Note that we do not require that $K_\infty$ is Galois over $k$.
However, there exists a cyclic extension $F/k_n$ of degree $p^m$
satisfying $F \subset K_\infty$ and $F \cap k_n = k_\infty$
if $n$ is sufficiently large.
We shall observe the behavior of $| X (F' k_{n'}) |$ for
$n \leq n'$ and $k_n \subset F' \subset F$.

In this paper, we treat the case that $K_\infty / k_\infty$ satisfies the following condition.

\begin{itemize}
\item[(R)] Only finitely many primes of $k_\infty$ ramify in $K_\infty$,
and at least one prime ramifies in $K_\infty / k_\infty$.
\end{itemize}

We remark that a prime not lying above $p$ can ramify in $K_\infty / k_\infty$.
Moreover, it can be occurred that all ramifying primes are not lying above $p$.
At this stage, 
we do not concern whether a prime lying above $p$ ramifies or not.

\begin{definition}
Assume that $K_\infty / k_\infty$ satisfies (R).
In this case, all primes ramified in $K_\infty / k_\infty$ are totally ramified in $K_\infty / K_m$
if $m$ is sufficiently large.
We denote by $e_1$ the minimal $m$ satisfying this property.
\end{definition}

We put $H = \Gal (K_\infty / k_\infty)$.
For a non-negative integer $m$, let $K_m$ be the ``$m$th layer'' of $K_\infty/ k_\infty$
(i.e., the unique intermediate field satisfying $[ K_m : k_\infty ]= p^m$).
We denote by $\Lambda_H$ the completed group ring $\mathbb{Z}_p [[ H ]]$.
We can see that $X (K_\infty)$ is a compact $\Lambda_H$-module.

Note that if $X (k_\infty)$ is finitely generated over $\mathbb{Z}_p$
(i.e., $\mu (k_\infty /k) =0$),
then $X (K_\infty)$ is a finitely generated $\Lambda_H$-module.
On the other hand, if $X (K_\infty)$ is a finitely generated $\Lambda_H$-module,
then all $X (K_m)$ must be finitely generated over $\mathbb{Z}_p$.
(For the proof of these facts, see, e.g., the argument given in the proof of \cite[Theorem 2]{Blo}.)
We also recall that if $k/\mathbb{Q}$ is an abelian extension 
then $\mu (k_\infty /k)=0$ (Ferrero-Washington \cite{FW}).

In this paper, we mainly treat the case that $K_\infty / k_\infty$ satisfies the following condition.

\begin{itemize}
\item[(T)] $X (K_\infty)$ is a finitely generated \textbf{torsion} $\Lambda_H$-module.
\end{itemize}

We remark that there is an example of $K_\infty  / k_\infty$ which satisfies (R) 
and $\mu (k_\infty /k) =0$ but does not satisfy (T) (see Hachimori-Sharifi \cite{HS}).
First, we introduce the following:

\begin{propA}
Assume that $K_\infty / k_\infty$ satisfies the conditions (R) and (T).
For a non-negative integer $m$,
we denote by $C_m$ the $\mathbb{Z}_p$-torsion subgroup of $X (K_m)$
(recall that $X (K_m)$ is finitely generated over $\mathbb{Z}_p$ in this case).
Then there are non-negative integers $\lambda_1$, $\mu_1$, $\lambda_2$ and
an integer $\nu_1$ such that the following holds for all sufficiently large $m$.
\begin{itemize}
\item[(i)] $|C_m| = p^{\lambda_1 n + \mu_1 p^m + \nu_1}$, and
\item[(ii)] $\rankzp X (K_m) = \lambda_2$.
\end{itemize}
\end{propA}

Here, for a $\mathbb{Z}_p$-module $A$, we denote by $\rankzp A$
the dimension of the $\mathbb{Q}_p$-vector space $A \otimes_{\mathbb{Z}_p} \mathbb{Q}_p$.
Proposition A seems almost well known (see Section \ref{sec_propA}).

\begin{definition}
Let $m_0$ be the minimal non-negative integer which satisfies $m_0 \geq e_1$, and
both (i), (ii) of Proposition A hold for all $m \geq m_0$.
\end{definition}

Next, we give a formula which describes the $p$-part of the class number
of certain intermediate fields of $K_\infty/ k$.
Recall that we do not assume that $K_\infty /k$ is a Galois extension.
However, as already noted a similar fact above,
there is an integer $n$ such that $K_m / k_n$ is an abelian extension
since $K_m / k_\infty$ is a finite abelian extension
($\Gal (K_m / k_n)$ is isomorphic to $\mathbb{Z} / p^m \times \mathbb{Z}_p$).

\begin{thm}\label{main_thm}
Assume that $K_\infty / k_\infty$ satisfies the conditions (R) and (T).
Fix an integer $m_1$ greater than $m_0$.
Let $n^*$ be the minimal non-negative integer such that $K_{m_1} / k_{n^*}$ is
an abelian extension.
Take a cyclic extension $F_{m_1} / k_{n^*}$ of degree $p^{m_1}$ such that
$F_{m_1} k_\infty = K_{m_1}$.
For an integer $m$ which satisfies $0 \leq m \leq m_1$, let $F_m$ be the
unique intermediate field of $F_{m_1} / k_{n^*}$ whose degree over $k_{n^*}$
is $p^m$.
Let $\lambda_1$, $\mu_1$, $\lambda_2$ be the non-negative
integers appeared in Proposition A.
Then there are an integer $\nu$ (depend only on $K_\infty /k$),
and a non-negative integer $n_0$ (depend on $F_{m_1} /k_{n^*}$)
such that
\[ |X (F_m k_n)| = p^{\lambda_1 m + \mu_1 p^m + \lambda_2 n + \nu} \]
holds for all integers $m,n$ satisfying
$m_0 \leq m \leq m_1$ and $n_0 \leq n$.
\end{thm}

The above theorem can be proven by only using the tools
given in the proof of Iwasawa's class number formula (see, e.g., \cite[Chapter 13]{Was}).
In Section \ref{sec_proof}, we will give remarks on Proposition A and prove this theorem.

In Theorem \ref{main_thm}, if every prime lying above $p$ is unramified in
$K_\infty / k_\infty$, then $F_{m_1} k_n /k_n$ is tamely ramified
for all sufficiently large $n$
(see, e.g., \cite[Exercise 13.11]{Was}).
In this case, we can obtain a tamely ramified cyclic extension of degree
$p^{m_1}$ whose behavior of the $p$-part of the class number of intermediate fields
is similar to that of Iwasawa's class number formula.
We also note that $m_1$ can be taken arbitrary large.

\begin{cor}\label{main_cor}
Let the notation be as in Theorem \ref{main_thm}.
Assume that $K_\infty / k_\infty$ satisfies the conditions (R) and (T), and
every prime lying above $p$ is unramified in $K_\infty / k_\infty$.
Then, for every fixed sufficiently large integer $n$,
$F_{m_1} k_n /k_n$ is tamely ramified and the formula
\[ |X (F_m k_n)| = p^{\lambda_1 m + \mu_1 p^m + \nu'} \]
holds for all $m$ satisfying $m_0 \leq m \leq m_1$.
The above constant $\nu'$ is independent on $m$.
\end{cor}

We shall give more remarks on the above theorem.
When $K_\infty / k$ is a $\mathbb{Z}_p^2$-extension, there is a
$\mathbb{Z}_p$-extension $F'_\infty / k_n$ such that $F'_\infty \subset K_\infty$
and the $m_1$th layer of $F'_\infty / k_n$ is $F_{m_1} k_n$.
In this situation, the Iwasawa invariants
$\lambda (F'_\infty / k_n)$, $\mu (F'_\infty / k_n)$ of $F'_\infty / k_n$ also exist.
However, the definition of our invariants $\lambda_1$, $\mu_1$ is slightly
different from that of the usual Iwasawa invariants.
Similarly, the definition of our invariants is also different from the
invariants stated in Corollary 3.4 of Lei's paper \cite{Lei} for 
$\mathbb{Z}_p \rtimes \mathbb{Z}_p$-extensions.
We will discuss about them and mention more several topics concerning known results 
in Section \ref{sec_remarks}.

Section \ref{sec_examples} is another main contents of this paper.
Concerning the above question,
we will give examples satisfying (R) and (T)
for the case that $K_\infty / k_\infty$ is unramified at
every prime lying above $p$.
(Although we do not deal in detail here, the other case seems also significant.
See Section \ref{sec_Galois_exam}.)
We can find an example satisfying $\lambda_1 = \mu_1 = \lambda_2 =0$ but $\nu \neq 0$
from a known result (Example \ref{ex_MO}).
We also give a lower bound of $\lambda_1$ (Proposition \ref{lower_lambda_1}).
To obtain this bound, we use the fact that the primes lying above $p$ are unramified in 
$K_\infty /k_\infty$.
We can construct an example $K_\infty / \mathbb{Q}$
such that $\lambda_1$ is larger than 
any given integer under the strict assumption that Greenberg's conjecture
holds for every intermediate algebraic number field (Example \ref{ex_GC}).
Alternatively, by using the following result, 
we can obtain an example with $\lambda_1 >0$ under a somewhat mild condition.

\begin{prop}\label{prop_Q}
We put $p=2$, and
we denote by $\mathbb{B}_\infty / \mathbb{Q}$ the cyclotomic $\mathbb{Z}_2$-extension.
Let $q_1$, $q_2$ be prime numbers satisfying $q_1 \equiv q_2 \equiv 7 \pmod{8}$.
Assume also that the number of primes of $q_1$ in $\mathbb{B}_\infty$
coincides with that of $q_2$.
Let $S$ be the set of primes of $\mathbb{B}_\infty$ lying above $\{ q_1, q_2 \}$.
Then the following holds:\\
(i) There is a $\mathbb{Z}_2$-extension $K_\infty / \mathbb{B}_\infty$ which is
unramified outside $S$ and satisfies $K_1 = \mathbb{B}_\infty (\sqrt{q_1 q_2})$.
(This $K_\infty / \mathbb{B}_\infty$ satisfies (R).) \\
(ii) Take $K_\infty$ as stated in (i).
If the Iwasawa $\lambda$-invariant of the
cyclotomic $\mathbb{Z}_2$-extension of $\mathbb{Q} (\sqrt{q_1 q_2})$ is zero,
then $K_\infty / \mathbb{B}_\infty$ satisfies (T), and moreover $\lambda_1 >0$.
\end{prop}

In the situation of the above proposition, it is expected that $\lambda_2 = 0$  
(see Section \ref{sec_GC}).
We can also obtain examples satisfying $\lambda_2 >0$ by using the following proposition 
(see also Section \ref{sec_Galois_exam} for preceding results).

\begin{prop}\label{prop_imag}
We put $p=2$.
Let $m$ be a positive odd square-free integer, and put $k = \mathbb{Q} (\sqrt{-m})$.
We assume that $X (k_\infty)$ is not trivial and $\mathbb{Z}_2$-torsion free.
Let $q$ be a prime number which does not divide $m$ and satisfies
$q \equiv 3 \pmod{8}$.
Let $S$ be the set of primes of $k_\infty$ lying above $q$.
Then the following holds:\\
(i) There is a $\mathbb{Z}_2$-extension $K_\infty / k_\infty$ which is
unramified outside $S$ and satisfies $K_1 = k_\infty (\sqrt{-q})$.
(This $K_\infty / k_\infty$ satisfies (R).) \\
(ii) Take $K_\infty$ as stated in (i).
If the Iwasawa $\lambda$-invariant of
the cyclotomic $\mathbb{Z}_2$-extension of $\mathbb{Q} (\sqrt{m q})$ is zero,
then $K_\infty / k_\infty$ satisfies (T), and moreover $\lambda_2 >0$.
\end{prop}

We also give an explicit example for Proposition \ref{prop_Q} (Example \ref{ex_Q}), 
and examples for Proposition \ref{prop_imag} (Examples \ref{ex_imag}, \ref{ex_imag_f}).
Moreover, as a special case of Proposition \ref{prop_imag}, we can obtain the 
following:

\begin{prop}\label{prop_imag2}
We put $p=2$.
Let $\ell$, $q$ be distinct prime numbers satisfying $\ell \equiv 7 \pmod{8}$ 
and $q \equiv 3 \pmod{8}$.
We denote by $r$ the number of primes of $\mathbb{B}_\infty$ lying above $\ell$ 
($r \geq 2$ in this situation).
Put $k =\mathbb{Q} (\sqrt{-\ell})$.
Let $S$ be the set of primes of $k_\infty$ lying above $q$.
Take a $\mathbb{Z}_2$-extension $K_\infty / k_\infty$ which is
unramified outside $S$ and satisfies $K_1 = k_\infty (\sqrt{-q})$
(such a $\mathbb{Z}_2$-extension exists by Proposition \ref{prop_imag} (i)).
Then 
\[ X (K_\infty) \cong \mathbb{Z}_2^{r-1} \]
as a $\mathbb{Z}_2$-module, 
and $\lambda_1 = \mu_1=0$, $\lambda_2 = r-1$.
\end{prop}

This can be shown by using the method given in the proof of \cite[Theorem 4.4]{MY_arch}.
Note that we can take $\ell$ such that $r$ is larger than
any given positive integer.

\section{Proof of Theorem \ref{main_thm}}\label{sec_proof}

\subsection{Remarks on Proposition A}\label{sec_propA}
Before stating the proof of Theorem \ref{main_thm}, we give several remarks
about $K_\infty / k_\infty$.

Assume that $K_\infty / k_\infty$ satisfies (R).
Fix a topological generator $h$ of $H$.
We recall the well known elements $\omega_n$ and $\nu_{m,n}$ of $\Lambda_H$:
\[ \omega_n = h^{p^n} - 1, \quad \nu_{m,n} = \omega_m / \omega_n \]
where $m$, $n$ are non-negative integers satisfying $m > n$.
We also note that there is a $\Lambda_H$-submodule $\mathcal{Y}_{e_1}$ of $X (K_\infty)$
such that
\[ X (K_m ) \cong X (K_\infty) / \nu_{m,e_1} \mathcal{Y}_{e_1} \]
holds for all $m \geq e_1$.
This fact can be shown by using the well known argument given in a proof of
Iwasawa's class number formula (cf. \cite{Blo} for the case of $\mathbb{Z}_p^2$-extensions).

Proposition A is also obtained by slightly modifying the argument given in a proof of
Iwasawa's class number formula, and we do not give a detailed proof here.
In fact, we can find that almost the same result is stated in Greenberg \cite{Gre99} 
(see pp.79--80).
The key points are the existence of an integer $e_2 (> e_1)$ such that
$\nu_{e_2,e_1} \mathcal{Y}_{e_1} / \nu_{m,e_1} \mathcal{Y}_{e_1}$ is finite for all $m > e_2$,
and the exact sequence
\[ 0 \to \nu_{e_2,e_1} \mathcal{Y}_{e_1} / \nu_{m,e_1} \mathcal{Y}_{e_1}
\to X (K_\infty) / \nu_{m,e_1} \mathcal{Y}_{e_1}
\to X (K_\infty) / \nu_{e_2,e_1} \mathcal{Y}_{e_1}
\to 0. \]
The proposition follows by using the fact that
\[ |\nu_{e_2,e_1} \mathcal{Y}_{e_1} / \nu_{m,e_1} \mathcal{Y}_{e_1}| =
p^{\lambda_1 m + \mu_1 p^m + \nu'_1} \]
for all sufficiently large $m$, where $\lambda_1$ (resp. $\mu_1$) is the
$\lambda$-invariant (resp. $\mu$-invariant) of $\nu_{e_2,e_1} \mathcal{Y}_{e_1}$
as a finitely generated torsion $\Lambda_H$-module and $\nu'_1$ is a constant
(see, e.g., \cite[Proposition 13.19 and Lemma 13.21]{Was}).
We also remark that $\lambda_1 + \lambda_2$ (resp $\mu_1$) is equal to the $\lambda$-invariant
(resp. $\mu$-invariant) of $X (K_\infty)$ as a finitely generated torsion $\Lambda_H$-module.

\subsection{Preliminaries on the proof of Theorem \ref{main_thm}}
In this and the next subsections, $m$ denotes an integer satisfying $0 \leq m \leq m_1$ 
and $N$ denotes a non-negative integer.
We note that $K_m = F_m k_\infty$.
Take a non-negative integer $e_3$ satisfying that
every prime lying above $p$ is totally ramified in $K_m / F_m k_{e_3}$ for all $m$.
We put $\mathcal{F}_{m, N} = F_m k_{e_3 + N}$.
Then, $\mathcal{F}_{m, N}$ is the $N$th layer of the cyclotomic $\mathbb{Z}_p$-extension
$K_m / \mathcal{F}_{m, 0}$.
We put $\Gamma = \Gal (F_{m_1} / \mathcal{F}_{m_1, 0})$,
and let $\gamma$ be a fixed topological generator of $\Gamma$.
Note that $\Gamma$ is isomorphic to $\Gal (F_{m} / \mathcal{F}_{m, 0})$
for every $0 \leq m < m_1$ via the restriction map, and we shall identify them.
(We use the same symbol $\gamma$ as an element of $\Gal (F_{m} / \mathcal{F}_{m, 0})$.)
We put $\Lambda_\Gamma = \mathbb{Z}_p [[\Gamma]]$.
It is well known that $X (K_m)$ is a finitely generated torsion $\Lambda_\Gamma$-module.

Put $\nu_N = 1+ \gamma + \gamma^2 + \cdots + \gamma^{p^N}$.
Let $Y^{(m_1)}$ be the $\Lambda_\Gamma$-submodule of $X (K_{m_1})$
defined in, e.g., \cite[Chapter 13, Lemma 13.15]{Was}.
That is, $Y^{(m_1)}$ satisfies
\[ X (K_{m_1})/ \nu_N Y^{(m_1)} \cong X (\mathcal{F}_{m_1, N}) \]
for all $N \geq 0$.
In the following, we briefly recall the definition of $Y^{(m_1)}$.
Let $\mathfrak{P}^{(m_1)}_1, \ldots, \mathfrak{P}^{(m_1)}_t$ be the primes of
$\mathcal{F}_{m_1, 0}$ lying above $p$.
For $i =1, \ldots, t$, we denote by $I^{(m_1)}_i$ the inertia subgroup
of $\Gal (L (K_{m_1}) /\mathcal{F}_{m_1, 0})$ for a prime lying above $\mathfrak{P}^{(m_1)}_i$.
We note that the restriction map induces an isomorphism 
$I^{(m_1)}_i \cong \Gamma$.
Take the topological generator $\sigma^{(m_1)}_i$ whose restriction image on $K_{m_1}$
coincides with $\gamma$.
It is known that there is a unique element $\gamma^{(m_1)}_i$ of 
$X (K_{m_1})$ such that $\sigma^{(m_1)}_i = \gamma^{(m_1)}_i \sigma^{(m_1)}_1$.
Then, $Y^{(m_1)}$ is the $\Lambda_\Gamma$-submodule generated by
$\gamma^{(m_1)}_2, \ldots, \gamma^{(m_1)}_t$, and $(\gamma -1) X (K_{m_1})$.

We put $\overline{X} (K_{m_1}) = X (K_{m_1}) / C_{m_1}$.
Then $\overline{X} (K_{m_1})$ is also a
finitely generated torsion $\Lambda_\Gamma$-module.
Let $\overline{Y}^{(m_1)}$ be the restriction image of $Y^{(m_1)}$ in
$\overline{X} (K_{m_1})$.

\begin{lem}\label{lem_m1_bar}
For all sufficiently large $N$, we see that
$\nu_N Y^{(m_1)} \cong \nu_N \overline{Y}^{(m_1)}$ and the
exact sequence
\begin{equation}\label{C_m1_exact}
0 \to C_{m_1} \to X (\mathcal{F}_{m_1, N}) \to \overline{X} (K_{m_1}) / \nu_N \overline{Y}^{(m_1)}
\to 0 \end{equation}
holds.
\end{lem}

\begin{proof}
The kernel of the restriction $Y^{(m_1)} \to\overline{Y}^{(m_1)}$ is finite
since it is contained in $C_{m_1}$.
Then we see that $\nu_N Y^{(m_1)} \cong \nu_N \overline{Y}^{(m_1)}$ if $N$ is sufficiently large.
From this and the isomorphism
$X (K_{m_1})/ \nu_N Y^{(m_1)} \cong X (\mathcal{F}_{m_1, N})$, we
also obtain \eqref{C_m1_exact}.
\end{proof}

Next, we fix an integer $m$ which satisfies $m_0 \leq m < m_1$.
Let $\mathfrak{P}^{(m)}_1, \ldots, \mathfrak{P}^{(m)}_s$ be the primes of
$\mathcal{F}_{m, 0}$ lying above $p$ (note that $s \leq t$).
Similar to the above,
take the inertia subgroup $I^{(m)}_j$ of $\Gal (L (K_{m}) /\mathcal{F}_{m, 0})$
for each $\mathfrak{P}^{(m)}_j$ ($j = 1,\ldots, s$).

For a fixed $\mathfrak{P}^{(m_1)}_i$, there is a prime
$\mathfrak{P}^{(m)}_j$ lying below it.
Let $\sigma^{(m)}_i$ be the restriction image of $\sigma^{(m_1)}_i$ on $L (K_{m})$.
Then $\sigma^{(m)}_i$ is contained in $I^{(m)}_j$, and its restriction on 
$K_m$ is $\gamma$.
Take the element $\gamma^{(m)}_i$ of $X (K_{m})$ such that 
$\sigma^{(m)}_i = \gamma^{(m)}_i \sigma^{(m)}_1$.
This $\gamma^{(m)}_i$ coincides with the restriction of $\gamma^{(m_1)}_i$ on $L (K_m)$.

Let $Y^{(m)}$ be the restriction image of $Y^{(m_1)}$ in $X (K_{m})$.
That is, $Y^{(m)}$ is generated by
$\gamma^{(m)}_2, \ldots, \gamma^{(m)}_t$, and $(\gamma -1) X (K_{m})$.
(We note that the restriction map $X (K_{m_1}) \to X (K_m)$ is surjective 
since at least one prime is totally ramified in $K_{m_1} / K_{m}$.)
If two primes $\mathfrak{P}^{(m_1)}_h$, $\mathfrak{P}^{(m_1)}_i$
are lying above the same prime $\mathfrak{P}^{(m)}_j$, then
$\gamma^{(m)}_h (\gamma^{(m)}_i)^{-1}$ is contained in $(\gamma -1) X (K_{m})$.
Considering this fact, we see that 
$Y^{(m)}$ coincides with the usual ``$Y$'' for $X (K_m)$ 
(in the sense of, e.g., \cite[Chapter 13]{Was}).
Hence, $X (K_{m})/ \nu_N Y^{(m)} \cong X (\mathcal{F}_{m, N})$ for all $N \geq 0$.

We put $\overline{X} (K_{m}) = X (K_{m}) / C_{m}$.
Similar to Lemma \ref{lem_m1_bar},
if $N$ is sufficiently large then
we also obtain the exact sequence
\begin{equation}\label{C_m_exact}
0 \to C_{m} \to X (\mathcal{F}_{m, N}) \to \overline{X} (K_{m}) / \nu_N \overline{Y}^{(m)} \to 0.
\end{equation}
Note that $\overline{X} (K_{m})$ is a free $\mathbb{Z}_p$-module of rank $\lambda_2$
for $m_0 \leq m \leq m_1$.
We see that
\[ \overline{X} (K_{m_1}) \cong \overline{X} (K_{m}) \]
via the restriction map 
(recall that at least one prime is totally ramified in $K_{m_1} / K_{m}$).
Moreover, we can see that the restriction image
$\overline{Y}^{(m_1)}$ in $\overline{X} (K_{m})$
coincides with $\overline{Y}^{(m)}$.

\begin{lem}\label{lem_m1_to_m}
Let $m$ be an integer satisfying $m_0 \leq m < m_1$. \\
(i) The restriction map induces an isomorphism
\[ \overline{Y}^{(m_1)} \cong \overline{Y}^{(m)}. \]
(ii) For all $N \geq 0$, we obtain an isomorphism
\[ \overline{X} (K_{m_1}) / \nu_N \overline{Y}^{(m_1)}
\cong \overline{X} (K_{m}) / \nu_N \overline{Y}^{(m)}. \]
\end{lem}

\begin{proof}
To show (i), it is sufficient to show that
$\overline{Y}^{(m_1)} \to \overline{Y}^{(m)}$ is injective.
This immediately follows from the fact that the restriction map
$\overline{X} (K_{m_1}) \to \overline{X} (K_{m})$ is an isomorphism.

By (i), we also obtain an isomorphism
$\nu _N \overline{Y}^{(m_1)} \cong \nu_N \overline{Y}^{(m)}$.
The assertion (ii) follows from this.
\end{proof}

\subsection{Finish of the proof of Theorem \ref{main_thm}}
Let the notation be as in the previous subsection.
We first consider the behavior of $| X (\mathcal{F}_{m_0,N}) |$.
Take the minimal integer $n^{\bullet}$ such that $K_{m_0}/k_{n^{\bullet}}$ is an abelian extension,
and a cyclic extension $F'_{m_0} /k_{n^{\bullet}}$ of degree $p^{m_0}$ satisfying
$K_{m_0} = F'_{m_0} k_{\infty}$.
By Iwasawa's class number formula, we see that
\[ | X (F'_{m_0} k_{n}) | = p^{\lambda_2 n + \nu'} \]
for all sufficiently large $n$ (with a constant $\nu'$).

\begin{rem}
The extension $F'_{m_0} /k_{n^{\bullet}}$ taken the above is not unique.
That is, there is another extension $F''_{m_0} /k_{n^{\bullet}}$ such that
$K_{m_0} = F''_{m_0} k_{\infty}$.
However, we see that $F''_{m_0} k_{n^{\bullet} + N'} = F'_{m_0} k_{n^{\bullet} + N'}$ when
$N' \geq m_0$.
Since $m_0$, $n^{\bullet}$ is dependent only on $K_\infty /k$,
the constant $\nu'$ is also dependent only on $K_\infty /k$.
By using the same argument, we also see that
$\mathcal{F}_{m_0,N} = F'_{m_0} k_{e_3 + N}$ if $N$ is sufficiently large.
\end{rem}

In the following, we assume that $N$ is sufficiently large.
From \eqref{C_m_exact}, we see that
\[ | X (\mathcal{F}_{m_0,N}) | = | \overline{X} (K_{m_0})/ \nu_N \overline{Y}^{(m_0)} | \cdot |C_{m_0}|. \]
Take the constant $\nu''$ such that $|C_{m_0}| = p^{\nu''}$
(actually, $\nu'' = \lambda_1 m_0 + \mu_1 p^{m_0} + \nu_1$).
By combining the above results, we see that
\[ | \overline{X} (K_{m_0})/ \nu_N \overline{Y}^{(m_0)} | = p^{\lambda_2 (e_3 + N) + \nu' - \nu''}. \]
Moreover, by Lemma \ref{lem_m1_to_m} (ii), we also see that
\[ | \overline{X} (K_{m})/ \nu_N \overline{Y}^{(m)} | = p^{\lambda_2 (e_3 + N) + \nu' - \nu''}. \]
for all $m_0 < m \leq m_1$.
Using \eqref{C_m1_exact}, \eqref{C_m_exact}, and Proposition A (i), we finally see that
\begin{equation}\label{final_eq}
| X (\mathcal{F}_{m,N}) | = |C_m| \cdot | \overline{X} (K_{m})/ \nu_N \overline{Y}^{(m)} |
= p^{\lambda_1 m + \mu_1 p^{m} + \lambda_2 (e_3 + N) + \nu' - \nu'' + \nu_1}
\end{equation}
for all $m_0 \leq m \leq m_1$.

We put $\nu = \nu' - \nu'' + \nu_1$.
This $\nu$ is dependent only on $K_\infty /k$ because
$m_0$, $\nu'$, $\nu''$, $\nu_1$ are so.
Recall that $\mathcal{F}_{m,N} = F_m k_{e_3 + N}$, and then
Theorem \ref{main_thm} follows from \eqref{final_eq}.

\section{Several remarks on Theorem \ref{main_thm}}\label{sec_remarks}

\subsection{Existing formulae and our formula}
First, assume that $K_\infty / k$ is a $\mathbb{Z}_p^2$-extension.
Then, as already mentioned in Section \ref{introduction},
there is a $\mathbb{Z}_p$-extension $F'_\infty / k_n$ inside $K_\infty$
such that $F_{m_1} k_n$ is its $m_1$th layer
We denote by $F'_m$ the $m$th layer of $F'_\infty / k_n$, and
recall that
\[ |X (F'_m)| = p^{\lambda (F'_\infty / k_n) m + \mu (F'_\infty / k_n) p^m + \nu (F'_\infty / k_n)} \]
if $m$ is sufficiently large.
On the other hand, the formula given in Corollary \ref{main_cor} describes the behavior of $| X (F'_m)|$
for relatively small $m$.
Moreover, the definition of $\lambda_1$, $\mu_1$ is slightly different
from $\lambda (F'_\infty / k_n)$, $\mu (F'_\infty / k_n)$.
In the following, we shall explain about this.

For simplicity, we assume that there is only one prime of $k$ lying above $p$,
and it is totally ramified in $K_\infty$.
Then, the usial Iwasawa invariants
$\lambda (F'_\infty / k_n)$, $\mu (F'_\infty / k_n)$ comes from the structure of
the coinvariant quotient $X (K_\infty)_{\Gal (K_\infty / F'_\infty)}$
as a $\mathbb{Z}_p [[ \Gal (F'_\infty / k_n)]]$-module.
Whereas, our invariants $\lambda_1$, $\mu_1$ comes from the
structure of $X (K_\infty)$ itself.
Hence it may possible that our $\lambda_1$ (resp. $\mu_1$) does not coincide with
$\lambda (F'_\infty / k_n)$ (resp. $\mu (F'_\infty / k_n)$) in general.
We also mention that there are infinitely many $\mathbb{Z}_p$-extensions
over $k_n$ inside $K_\infty$ being $F_{m_1} k_n$ its $m_1$th layer.

Next, assume that $K_\infty / k$ is a $\mathbb{Z}_p \rtimes \mathbb{Z}_p$-extension.
A.~Lei \cite{Lei} gives a Iwasawa type formula for a
certain sequence of intermediate fields.
(Actually, he gives the formula for more general $p$-adic Lie extensions.)
The formula given in \cite[Corollary 3.4]{Lei} for the case that $d=2$ seems close to ours
(note that Lei's formula holds without assuming (T)).
However, similar to the above case, Lei's formula covers for sufficiently large $m$.
In particular, the field appears in Lei's formula is a non-Galois extension
over $k_n$ of degree $p^m$ in general if $m$ is large (see \cite[p.348]{Lei}).
Recall that our formula covers the case that $m$ is 
relatively small (and the fields are cyclic extensions over $k_n$).
Furthermore, by the similar reason to the above,
the definition of Lei's invariants $\lambda_H (\mathcal{X}_{\Gamma_n})$,
$\mu_H (\mathcal{X}_{\Gamma_n})$ is different from our $\lambda_1$, $\mu_1$.

\subsection{More remarks about preceding studies}\label{sec_Galois_exam}
When $K_\infty / k_\infty$ is a $\mathbb{Z}_p \rtimes \mathbb{Z}_p$-extension,
our condition (T) relates the pseudo-nullity of $X (K_\infty)$ as
a $\mathbb{Z}_p [[\Gal (K_\infty / k)]]$-module.
For the details, see, e.g., Hachimori-Sharifi \cite{HS}.

For example, let $k$ be an imaginary quadratic field, and
$K_\infty /k$ the unique $\mathbb{Z}_p^2$-extension.
In this case, Greenberg's generalized conjecture \cite[Conjecture (3.5)]{Gre01} (GGC)
says that $X (K_\infty)$ is a pseudo-null
$\mathbb{Z}_p [[\Gal (K_\infty / k_\infty)]]$-module.
In this situation, GGC is equivalent to the condition 
that $X (K_\infty)$ is $\Lambda_H$-torsion
(see \cite{HS}, \cite{I07}).
Moreover, if $p$ does not split in $k$, then $K_\infty / k_\infty$ satisfies (R).
Several studies of GGC for this case exist.
See, e.g., \cite{Min}, \cite{I07}, \cite{Fuj}.
In particular, one can find examples satisfying $\lambda_2 > 0$ in
these literature.

There are also several studies on the pseudo-nullity of
$X (K_\infty)$ for the case of non-commutative $\mathbb{Z}_p \rtimes \mathbb{Z}_p$-extensions.
\begin{itemize}
\item Li-Ouyang-Xu-Zhang \cite{LOXZ} :
They gives the results on the class numbers of
$\mathbb{Q} (\zeta_{p^n}, \sqrt[p^m]{q})$.
In particular the finiteness of $X (\mathbb{Q} (\zeta_{p^\infty}, \sqrt[p^\infty]{q}))$
is shown for some cases.
\item Yamamoto \cite{Yamamoto}, Mizusawa-Yamamoto \cite{MY_tjm}, \cite{MY_arch}:
They gave certain $\mathbb{Z}_p \rtimes \mathbb{Z}_p$-extensions $K_\infty /k$
constructed by iterated extensions, 
and showed that it satisfies (R) and (T) for several specific cases.
In particular, some examples given by them satisfy $\lambda_2 >0$ 
(see \cite[Theorem 1.2 and Section 6]{Yamamoto} and \cite[Theorem 4.4, Remark 4.5, and Example 4.6]{MY_arch}).
\end{itemize}

In the next section, we will give more examples.
We remark that our Proposition \ref{prop_imag} treats a different situation 
from those mentioned above.

\section{Examples for the case that $K_\infty /k_\infty$ is unramified at the primes
lying above $p$}\label{sec_examples}

We shall define several notation used in this section.
Let $S$ be a finite set of non-archimedean primes of $k_\infty$.
In the following, we always assume that $S$ does not contain any prime lying above $p$.
We denote by $L_S (k_\infty) / k_\infty$ the maximal abelian pro-$p$ extension
unramified outside $S$, and put $X_S (k_\infty) = \Gal (L_S (k_\infty) / k_\infty)$.
This $X_S (k_\infty)$ is called a ``tamely ramified Iwasawa module'',
and several studies on $X_S (k_\infty)$ exist (e.g., \cite{IMO}).
Note that if $X (k_\infty)$ is finitely generated over $\mathbb{Z}_p$, then
$X_S (k_\infty)$ is also finitely generated over $\mathbb{Z}_p$ 
(see, e.g., \cite[p.1494]{IMO}).

We also denote $L_S (K_\infty) / K_\infty$ by the maximal abelian pro-$p$ extension 
unramified outside the primes lying above $S$, 
and put $X_S (K_\infty) = \Gal (L_S (K_\infty) / K_\infty)$.
We can see that $X_S (K_\infty)$ is a compact $\Lambda_H$-module.

\subsection{Examples obtained from known results}

The following lemma is useful to find examples of $K_\infty/k_\infty$ satisfying 
(R) and (T).

\begin{lem}\label{lem_S_ram}
Assume that $X (k_\infty)$ is finite and $\rankzp X_S (k_\infty) =1$.
Then there exists a unique $\mathbb{Z}_p$-extension $K_\infty / k_\infty$
unramified outside $S$.
Moreover this $K_\infty / k_\infty$ satisfies (R) and (T).
\end{lem}

\begin{proof}
The existence of $K_\infty / k_\infty$ and the validity of (R)
immediately follow from the assumption.

We shall show the validity of (T).
Since $X (K_\infty)$ is a quotient of $X_S (K_\infty)$,
it is sufficient to show that
$X_S (K_\infty)$ is a finitely generated torsion $\Lambda_H$-module.
Since 
\[ X_S (K_\infty)/ \omega_0 X_S (K_\infty) \cong \Gal (L_S (k_\infty) /K_\infty) \]
and $L_S (k_\infty) /K_\infty$ is a finite extension, 
the assertion follows from the well known argument using the structure theorem 
of finitely generated $\Lambda_H$-module 
(see, e.g., \cite[p.281]{Was}).
\end{proof}

\begin{example}[see \cite{IMO}]\label{ex_S_ram}
Assume that $p$ is odd.
We state typical examples of $k$ and $S$ satisfying the assumption of Lemma \ref{lem_S_ram}.
In the following list, $q, q_1, q_2$ are prime numbers ($q_1 \neq q_2$), 
and take $S$ as the set of primes of $k_\infty$ lying above $S_0$.
\begin{itemize}
\item[(a)] $k = \mathbb{Q}$.
$S_0 = \{ q_1, q_2 \}$ with
$q_i \equiv 1 \pmod{p}$ and $q_i \not\equiv 1 \pmod{p^2}$ for $i=1,2$.
\item[(b)] $k$ is an imaginary quadratic field satisfying the following:
$p$ does not split in $k$ and the class number of $k$ is prime to $p$.
$S_0 = \{ q \}$ with $q \equiv 1 \pmod{p}$, $q \not\equiv 1 \pmod{p^2}$,
and $q$ splits in $k$.
\item[(c)] $k$ is the same as given in (b).
$S_0 = \{ q \}$ with $q \equiv -1 \pmod{p}$, $q^2 \not\equiv 1 \pmod{p^2}$,
and $q$ is inert in $k$.
\end{itemize}
By using Theorem 1.1 or Theorem 1.4 of \cite{IMO}, we can see that the above cases satisfy the assumption
(we see that $X (k_\infty)$ is trivial by Iwasawa's theorem \cite{Iwa56}).
\begin{itemize}
\item We can see that $X (K_\infty)$ is trivial for the case (c).
Hence $\lambda_1 = \mu_1 = \lambda_2 = \nu =0$ for this case.
\item On the other hand, we can see that $X (K_\infty)$ is not trivial for the cases (a), (b).
In these cases, we can show that $L_S (k_\infty) /K_\infty$ is not trivial
(this follows from the results given in \cite{IMO}).
Moreover, since all primes of $S$ ramifies in $K_\infty / k_\infty$ and
the inertia subgroup of $L_S (k_\infty) /k_\infty$ for any prime of $S$ is pro-cyclic,
we see that $L_S (k_\infty) /K_\infty$ is unramified.
The non-triviality of $X (K_\infty)$ follows from this.
\end{itemize}
\end{example}

For the cases (a), (b) of the above Example \ref{ex_S_ram}, we need more consideration to
determine the invariants $\lambda_1$, $\mu_1$, $\lambda_2$.
However, we can determine them in a specific situation.

\begin{example}[see {\cite[Theorem 1]{MO}}]\label{ex_MO}
We give more consideration for the case (a) of Example \ref{ex_S_ram}.
We additionally suppose that the assumptions of \cite[Theorem 1]{MO} are satisfied.
Let $\widetilde{L}_S (k_\infty) / k_\infty$ be the maximal pro-$p$ extension 
unramified outside $S$.
By this theorem, we see that $\Gal (\widetilde{L}_S (k_\infty) / k_\infty)$ is 
isomorphic to $\mathbb{Z} / p^a \mathbb{Z} \rtimes \mathbb{Z}_p$ with some 
positive integer $a$.
This yields that $L_S (K_\infty)/ K_\infty$ is a finite extension.
Hence $L (K_\infty)/ K_\infty$ is also finite,
and then $\lambda_1 = \mu_1 = \lambda_2 = 0$.
On the other hand, from the fact that $X (K_\infty)$ is not trivial, we also see that $\nu >0$.
\end{example}

We mention that similar results to the above example for the case $p=2$ are also known.
See, e.g., \cite[Theorem 1]{Salle}, \cite[Theorem 2]{MO}.

\subsection{A lower bound of $\lambda_1$}
In \cite[p.266]{Gre76}, a lower bound of the $\lambda$-invariant of a 
$\mathbb{Z}_p$-extension of an algebraic number field is given.
We shall give a similar type result.

In this subsection, we always assume that $X_S (k_\infty)$ is 
finitely generated over $\mathbb{Z}_p$.
We put $d = \rankzp X_S (k_\infty)$.
Let $\overline{L}_S /k_\infty$ be the unique $\mathbb{Z}_p^d$-extension 
unramified outside $S$.
We denote by $S' (\subset S)$ the set of primes of $k_\infty$ which actually ramify in 
$\overline{L}_S /k_\infty$.

\begin{prop}\label{lower_lambda_1}
Let the notation be as in the previous paragraph, and suppose that $d \geq 1$.
Take a $\mathbb{Z}_p$-extension $K_\infty /k_\infty$ which is unramified outside $S'$.
Assume that $K_\infty /k_\infty$ satisfies (R) and 
all primes of $S'$ are totally ramified in $K_\infty /k_\infty$.
Assume also that $K_\infty /k_\infty$ satisfies (T).
Then the following inequality holds.
\[ \lambda_1 \geq d - 1 -\rankzp X (k_\infty). \]
\end{prop}

\begin{proof}
Recall that $S$ does not contain any prime lying above $p$.
Then, the inertia subgroup for any prime of $S'$ in $\Gal (\overline{L}_S /k_\infty)$ 
is pro-cyclic (recall the argument given in Example \ref{ex_S_ram}).
By the assumption, we see that $\overline{L}_S /K_\infty$ is unramified at every prime.
Hence $\overline{L}_S \subset L (K_\infty)$.

We put 
$\mathcal{C} := \Gal (L(K_\infty)/\overline{L}_S)$ and 
$\mathcal{F} := \Gal (\overline{L}_S /K_\infty)$.
These are finitely generated torsion $\Lambda_H$-modules. 
Since $\overline{L}_S /k_\infty$ is an abelian extension, 
$H$ acts trivially on $\mathcal{F}$.
Let $\nu_{m,0}$ be the element of $\Lambda_H$ defined in Section \ref{sec_propA}.
By the assumption on ramification of primes, there is a 
$\Lambda_H$-submodule $\mathcal{Y}_0$ of $X (K_m)$ which satisfies 
\[ X (K_m) \cong X (K_\infty) / \nu_{m,0} \mathcal{Y}_0 \]
for all $m \geq 0$ (recall the fact stated in Section \ref{sec_propA}).

By applying the snake lemma for the following commutative diagram with exact rows 
\[ \begin{CD}
0 @>>> \mathcal{C} @>>> X (K_\infty) @>>> \mathcal{F} @>>> 0 \\
@.  @VV{\nu_{m,0}}V       @VV{\nu_{m,0}}V    @VV{\nu_{m,0}}V @. \\
0 @>>> \mathcal{C} @>>> X (K_\infty) @>>> \mathcal{F} @>>> 0, 
\end{CD} \]
we obtain the following exact sequence 
\[ \mathcal{F} [\nu_{m,0}] \to 
\mathcal{C}/ \nu_{m,0} \mathcal{C} \to X (K_\infty)/ \nu_{m,0}  X (K_\infty) 
\to \mathcal{F} / \nu_{m,0} \mathcal{F} \to 0 \]
for all $m$.
Here, $\mathcal{F} [\nu_{m,0}]$ is the kernel of the right vertical map.
Note that $\mathcal{F}$ is free of rank $d-1$ as a $\mathbb{Z}_p$-module, 
and $H$ acts trivially on $\mathcal{F}$.
From this, we can see that $\mathcal{F} [\nu_{m,0}]$ is trivial.
We also see that $\mathcal{F} / \nu_{m,0} \mathcal{F}$ is finite.
This implies that 
\[ \rankzp (X (K_\infty)/ \nu_{m,0}  X (K_\infty))
= \rankzp (\mathcal{C}/ \nu_{m,0} \mathcal{C}) \leq \rankzp \mathcal{C}. \]

Since $\rankzp X (K_\infty) = \lambda_1 + \lambda_2$, 
we see that $\rankzp \mathcal{C} = \lambda_1 + \lambda_2 - d +1$, 
and then  
\begin{equation}\label{inequality_X}
\rankzp (X (K_\infty)/ \nu_{m,0}  X (K_\infty)) \leq \lambda_1 + \lambda_2 - d +1.
\end{equation}

Next, we consider the following exact sequence 
\[ 0 \to \mathcal{Y}_0 \to X (K_\infty) \to X (K_\infty)/\mathcal{Y}_0 \to 0. \]
Similar to the previous argument, 
we obtain the following exact sequence
\begin{multline*}
(X (K_\infty)/\mathcal{Y}_0) [\nu_{m,0}] \to 
\mathcal{Y}_0 / \nu_{m,0} \mathcal{Y}_0 \to X (K_\infty)/ \nu_{m,0}  X (K_\infty) \\
\to (X (K_\infty)/\mathcal{Y}_0) / \nu_{m,0} (X (K_\infty)/\mathcal{Y}_0) \to 0. 
\end{multline*}
We note that $X (K_\infty)/\mathcal{Y}_0$ is isomorphic to $X (k_\infty)$.
By the assumption, $X (k_\infty)$ is finitely generated over $\mathbb{Z}_p$.
Since $H$ acts trivially on $X (K_\infty)/\mathcal{Y}_0$, 
we see that both 
$(X (K_\infty)/\mathcal{Y}_0) [\nu_{m,0}]$ and 
$(X (K_\infty)/\mathcal{Y}_0) / \nu_{m,0} (X (K_\infty)/\mathcal{Y}_0)$ 
are finite.
Hence 
\begin{equation}\label{equation_Y}
\rankzp (\mathcal{Y}_0 / \nu_{m,0} \mathcal{Y}_0) = 
\rankzp (X (K_\infty)/ \nu_{m,0}  X (K_\infty)). 
\end{equation}

Finally, by using the exact sequence 
\[ 0 \to \mathcal{Y}_0 / \nu_{m,0} \mathcal{Y}_0 \to 
X (K_\infty) / \nu_{m,0} \mathcal{Y}_0 \to X (K_\infty) / \mathcal{Y}_0 \to 0, \]
\eqref{inequality_X}, and \eqref{equation_Y}, we see that 
\begin{eqnarray*}
\rankzp X (K_m) & = & \rankzp (X (K_\infty) / \nu_{m,0} \mathcal{Y}_0) \\
 & = & \rankzp (\mathcal{Y}_0 / \nu_{m,0} \mathcal{Y}_0) + \rankzp (X (K_\infty) / \mathcal{Y}_0) \\
 & = & \rankzp (X (K_\infty)/ \nu_{m,0}  X (K_\infty)) + \rankzp X (k_\infty) \\
 & \leq &  \lambda_1 + \lambda_2 - d +1 + \rankzp X (k_\infty).
\end{eqnarray*}
If $m$ is sufficiently large, then $\rankzp X (K_m) = \lambda_2$, and hence
\[ \lambda_2 \leq  \lambda_1 + \lambda_2 - d +1 + \rankzp X (k_\infty). \]
The assertion follows from this.
\end{proof}

\subsection{Example under a strict condition}\label{sec_GC}
First at all, we recall Greenberg's conjecture \cite{Gre76} (GC).
This conjecture states that
$\lambda (k'_\infty /k') = \mu (k'_\infty /k') = 0$ if $k'$ is a totally real
algebraic number field.
It is equivalent to that $X (k'_\infty)$ is finite.

In this paragraph, assume that $k$ is totally real 
and $K_\infty /k_\infty$ satisfies (R).
We note that $K_m$ is a cyclotomic $\mathbb{Z}_p$-extension of a certain intermediate field
$k'$ with finite degree over $k$.
This $k'$ is also totally real because no archimedean prime ramifies in $K_\infty /k$.
We can see that the following conditions are equivalent 
(cf. also \cite[Statement 4 of Theorem 1]{MY_tjm}):
\begin{itemize}
\item GC holds for every intermediate field of $K_\infty /k$ with finite degree, 
\item $X (K_m)$ is finite for every $m \geq 0$,
\item $K_\infty / k_\infty$ satisfies (T),  $\lambda_2 = 0$,
and $\lambda_1 = \rankzp X (K_\infty)$.
\end{itemize}
(As an additional remark, the proof of Theorem \ref{main_thm} under these equivalent conditions 
becomes much simpler.)
We will give an explicit example (cf. the case (a) of Example \ref{ex_S_ram}).

\begin{example}\label{ex_GC}
Assume that $p$ is odd, and take an integer $r \geq 3$.
Let $q_1, q_2, \ldots, q_r$ be distinct prime numbers satisfying
$q_i \equiv 1 \pmod{p}$ and  $q_i \not\equiv 1 \pmod{p^2}$
for $i = 1, 2, \ldots r$.
Let $S$ be the set of primes of $\mathbb{B}_\infty$ lying above
$\{q_1, \ldots, q_r \}$.
Then, by \cite[Theorem 1.1]{IMO}, we see that $\rankzp X_S (\mathbb{B}_\infty ) = r-1$.
We can take a $\mathbb{Z}_p$-extension $K_\infty / \mathbb{B}_\infty$ such that all primes of
$S$ are ramified.
Assume that GC holds for all intermediate fields of
$K_\infty / \mathbb{Q}$ with finite degree.
Then $K_\infty / \mathbb{B}_\infty$ satisfies (R) and (T).
Moreover, $\lambda_1 \geq r-2$ by Proposition \ref{lower_lambda_1} 
(recall that $X (\mathbb{B}_\infty)$ is trivial by Iwasawa's theorem \cite{Iwa56}).
\end{example}

\subsection{Proofs of Propositions \ref{prop_Q} and \ref{prop_imag} and examples}
We mention that Propositions \ref{prop_Q} and \ref{prop_imag}
are shown by using the method given in the previous study \cite{I07} on GGC.
We shall introduce the following criterion for the validity of (T).
This is an immediate consequence of the argument given in \cite{Blo}. 

\begin{lemB}[cf. the proof of {\cite[Corollary 1]{Blo}}]
Assume that $K_\infty /k_\infty$ satisfies (R), and
$e_1 =0$
(i.e., all primes which ramify in $K_\infty / k_\infty$ are totally ramified).
If $X (k_\infty)$ is finitely generated over $\mathbb{Z}_p$ and
\[ \rankzp X (K_m) - \rankzp X (k_\infty) < p^m - 1 \] 
with some $m \geq 1$, then $K_\infty /k_\infty$ satisfies (T).
\end{lemB}

\begin{proof}
We recall that 
$X (K_\infty)$ is a finitely generated $\Lambda_H$-module under the assumption of this proposition.
Let $\mathcal{Y}_0$ be the submodule of $X (K_m)$ stated in Section \ref{sec_propA}.
Since $\mathcal{Y}_0$ is also a finitely generated $\Lambda_H$-module,
there is a pseudo-isomorphism 
\[ \mathcal{Y}_0 \to \Lambda_H^{r} \oplus E, \]
where $r$ is a non-negative integer and $E$ is an elementary torsion 
$\Lambda_H$-module (see, e.g., \cite[p.351]{Was}).
By using the argument given in the proof of \cite[Corollary 1]{Blo}, 
we see that 
\[ \rankzp (\mathcal{Y}_0 / \nu_{m,0} \mathcal{Y}_0) 
= r (p^m -1) + C_m \]
for all $m$.
Here, $C_m$ is a non-negative integer which satisfies 
$C_{m+1} \geq C_m$ for every $m$.
We also note that 
\begin{eqnarray*}
\rankzp X (K_m) - \rankzp X (k_\infty) 
 & = & \rankzp (X (K_\infty) / \nu_{m,0} \mathcal{Y}_0) - 
\rankzp (X (K_\infty) / \mathcal{Y}_0) \\
 & = & \rankzp (\mathcal{Y}_0 / \nu_{m,0} \mathcal{Y}_0). 
\end{eqnarray*}
If $X (K_\infty)$ is not $\Lambda_H$-torsion, then 
$r \geq 1$, and moreover 
\[ \rankzp X (K_m) - \rankzp X (k_\infty) =
r (p^m -1) + C_m \geq p^m -1. \]
This yields the assertion of this lemma.
\end{proof}

In particular, if $\rankzp X (K_1) = \rankzp X (k_\infty)$, then (T) is sarisfied. 
This is an analog of \cite[Lemma 1]{I07}, however, the above lemma is 
more useful in general.

In the rest of this subsection, we fix $p=2$.

\begin{proof}[Proof of Proposition \ref{prop_Q}]
Let $r$ be the number of primes lying above $q_1$
(which is equal to that of $q_2$).
By the assumption, $r \geq 2$.
Using \cite[Theorem 1.1]{IMO} and \cite[Theorem 1.1]{I18}, we see that
$X_S (\mathbb{B}_\infty)$ is a free $\mathbb{Z}_2$-module of rank $r$.
Hence, every quadratic subextension of $L_S (\mathbb{B}_\infty) / \mathbb{B}_\infty$ 
is extendable to a $\mathbb{Z}_2$-extension contained in $L_S (\mathbb{B}_\infty)$.
Since $\mathbb{B}_\infty (\sqrt{q_1 q_2})$ is contained in $L_S (\mathbb{B}_\infty)$,
we obtain (i).

We shall show (ii).
Since $K_1 = \mathbb{B}_\infty (\sqrt{q_1 q_2})$, all primes in $S$ is totally ramified
in $K_\infty / \mathbb{B}_\infty$.
By the assumption, we see that $X (K_1)$ is finite.
Hence $K_\infty / \mathbb{B}_\infty$ satisfies (T) by Lemma B.
Finally, by using Proposition \ref{lower_lambda_1}, we see that 
$\lambda_1 \geq r-1 > 0$.
\end{proof}

\begin{example}\label{ex_Q}
Let the notation be as in Proposition \ref{prop_Q}.
Thanks to the computational result given by Pagani \cite{Pagani}, 
we know that the Iwasawa $\lambda$-invariant of the cyclotomic $\mathbb{Z}_2$-extension of
$\mathbb{Q} (\sqrt{q_1 q_2})$ is zero if $q_1 q_2 < 10000$. 
We will give an explicit example.
Put $q_1= 31$ and $q_2 = 223$, and
let $K_\infty / \mathbb{B}_\infty$ be a $\mathbb{Z}_2$-extension stated in (i) of
Proposition \ref{prop_Q}.
Then $K_\infty / \mathbb{B}_\infty$ satisfies (R) and (T).
We also see that
$\rankztwo X_S (\mathbb{B}_\infty ) = 8$ by \cite[Theorem 1.1]{IMO}.
Hence, in this case, we see that $\lambda_1 \geq 7$.
\end{example}

\begin{proof}[Proof of Proposition \ref{prop_imag}]
For (i), it is sufficient to show that $X_S (K_\infty)$ is a
finitely generated free $\mathbb{Z}_2$-module with positive rank.
The remaining part can be shown similarly as the proof of Preposition \ref{prop_Q} (i).

In the following, we use the arguments given in several previous
studies on ``tamely ramified Iwasawa modules''
(see, e.g., \cite{IMO}).
We remark that $X_S (k_\infty)$ is known to be finitely generated as a $\mathbb{Z}_2$-module,
hence we shall show it is free.
By the assumption on $q$, there are exactly two primes
$\mathfrak{q}_n$, $\mathfrak{q}'_n$of $k_n$ lying above $q$ if $n \geq 1$.
We denote by $O_{k_n}$ (resp. $O_{\mathbb{B}_n}$) the ring of integers of $k_n$
(resp. $\mathbb{B}_n$).
We also denote by $E_{k_n}$ (resp. $E_{\mathbb{B}_n}$) the group of units of $k_n$
(resp. $\mathbb{B}_n$).
In the following, we assume that $n \geq 1$.
By class field theory, we obtain the exact sequence
\[ E_{k_n} \otimes_{\mathbb{Z}} \mathbb{Z}_2 \to
(O_{k_n} / (q))^\times \otimes_{\mathbb{Z}} \mathbb{Z}_2 \to
X_S (k_n) \to X (k_n) \to 0, \]
where $X_S (k_n)$ is the Galois group of the maximal abelian $p$-extension
over $k_n$ unramified outside $\{ \mathfrak{q}_n$, $\mathfrak{q}'_n \}$.
We  also obtain the exact sequence
\[ E_{\mathbb{B}_n} \otimes_{\mathbb{Z}} \mathbb{Z}_2 \to
(O_{\mathbb{B}_n} / (q))^\times \otimes_{\mathbb{Z}} \mathbb{Z}_2 \to
X_S (\mathbb{B}_n) \to 0, \]
where $X_S (\mathbb{B}_n)$ is defined similarly as above.
Since $q \equiv 3 \pmod{8}$, we see that $X_S (\mathbb{B}_n)$ is 
trivial for all $n$ (e.g., \cite[Proposition 2.5]{Salle}).
Hence, the map $E_{\mathbb{B}_n} \otimes_{\mathbb{Z}} \mathbb{Z}_2 \to
(O_{\mathbb{B}_n} / (q))^\times \otimes_{\mathbb{Z}} \mathbb{Z}_2$ is surjective.
We note that
\[ (O_{k_n} / (q))^\times \otimes_{\mathbb{Z}} \mathbb{Z}_2
\cong \mathbb{Z} / 2^a \mathbb{Z} \oplus \mathbb{Z} / 2^a \mathbb{Z}, \quad
(O_{\mathbb{B}_n} / (q))^\times \otimes_{\mathbb{Z}} \mathbb{Z}_2
\cong \mathbb{Z} / 2^a \mathbb{Z} \]
as abelian groups with some positive integer $a$.
By combining these facts, we see that the cokernel of
\[ E_{k_n} \otimes_{\mathbb{Z}} \mathbb{Z}_2 \to
(O_{k_n} / (q))^\times \otimes_{\mathbb{Z}} \mathbb{Z}_2 \]
is a cyclic group.
Let $\lambda_0$ be the $\mathbb{Z}_2$-rank of $X (k_\infty)$.
By the assumption, the $2$-rank of $X (k_\infty)$ is also $\lambda_0$.
(For a finitely generated $\mathbb{Z}_2$-module $A$, we denote by its
$2$-rank the $\mathbb{F}_2$-dimension of $A /2 A$.)
We note that the $2$-rank of $X (k_n)$ is equal to $\lambda_0$
if $n$ is sufficiently large.
Hence the $2$-rank of $X_S (k_n)$ is at most $\lambda_0 + 1$.
We then see that the $2$-rank of $X_S (k_\infty)$ is at most $\lambda_0 + 1$.
On the other hand, by using \cite[Theorem 1.4]{IMO} (see also \cite[Theorem 2.1]{Salle}),
we see that 
$\rankztwo X_S (k_\infty)$ is just equal to $\lambda_0 + 1$.
Hence $X_S (k_\infty)$ must be a finitely generated free $\mathbb{Z}_2$-module of
rank $\lambda_0 + 1$.

We shall show (ii).
Recall that $\lambda_0$ is the $\mathbb{Z}_2$-rank of $\rankztwo X (k_\infty)$.
We claim that $\rankztwo X (K_1) = \lambda_0$.
From this claim and Lemma B,
we can see that $K_\infty / k_\infty$ satisfies (T).

To show the above claim,
we recall the following known equality
\[ \rankztwo X (K_1) = \rankztwo X (k_\infty) +
\rankztwo X (\mathbb{B}_\infty (\sqrt{-q})) + \rankztwo X (\mathbb{B}_\infty (\sqrt{mq})). \]
(cf., e.g., \cite[Proof of Proposition 1]{I07}).
It is also known that $\rankztwo X (\mathbb{B}_\infty (\sqrt{-q})) = 0$
(\cite{Ferrero}, \cite{Kida}).
By our assumption on the $\lambda$-invariant, we see that $\rankztwo X (\mathbb{B}_\infty (\sqrt{mq})) =0$.
Hence the claim follows.

Since we assumed that $X (k_\infty)$ is not finite, we also see that
$\lambda_2 > 0$.
\end{proof}

\begin{rem}\label{rem_F-K}
As already used in the above proof,
we can compute the $\lambda$-invariant of the cyclotomic $\mathbb{Z}_2$-extension
of an explicitly given imaginary quadratic field $k$ by using the formula of Ferrero and Kida
(\cite{Ferrero}, \cite{Kida}).
We also note that if $2$ is unramified in $k$, then $X (k_\infty)$ is $\mathbb{Z}_2$-torsion free
(see, e.g., \cite[Theorem 1]{Kida}).
\end{rem}

\begin{example}\label{ex_imag}
Let the notation be as in Proposition \ref{prop_imag}.
As similar to the case of Example \ref{ex_Q}, it is known that 
the Iwasawa $\lambda$-invariant of the cyclotomic $\mathbb{Z}_2$-extension 
of $\mathbb{Q} (\sqrt{mq})$ is zero if $mq < 10000$ (Pagani \cite{Pagani}).
As an explicit example, we shall treat the case that $m=1463 = 7 \cdot 11 \cdot 19$ and $q=3$.
In this case, 
we can show that $X (k_\infty) \cong \mathbb{Z}_2^3$ as a $\mathbb{Z}_2$-module 
(see the above remark).
Moreover, since $mq= 4389$, we see that $X (\mathbb{B}_\infty (\sqrt{mq}))$ is finite.
Let $K_\infty / k_\infty$ be a $\mathbb{Z}_2$-extension stated in Proposition \ref{prop_imag} (i).
Then $K_\infty / k_\infty$ satisfies (R) and (T), and moreover $\lambda_2 \geq 3$.
\end{example}

\begin{example}\label{ex_imag_f}
Let the notation be as in Proposition \ref{prop_imag}.
We put $m= \ell_a \ell_b$, where $\ell_a, \ell_b$ are distinct prime numbers 
satisfying $\ell_a \equiv 7 \pmod{8}$ and $\ell_b \equiv 5 \pmod{8}$. 
Then, as a $\mathbb{Z}_2$-module, $X (k_\infty)$ is isomorphic to $\mathbb{Z}_2^c$  
with some positive integer $c$ (Remark \ref{rem_F-K}).
In this case, 
we see that $X (\mathbb{B}_\infty (\sqrt{mq}))$ is non-trivial and finite 
by the result of Mouhib-Movahhedi \cite[Theorem 4.4]{MM}.
Take a $\mathbb{Z}_2$-extension $K_\infty / k_\infty$ stated in Proposition \ref{prop_imag} (i).
Then $K_\infty / k_\infty$ satisfies (R) and (T), and 
$\lambda_2 \geq c$.
\end{example}

\subsection{Proof of Proposition \ref{prop_imag2}}
We also fix $p=2$ in this subsection.
To show Proposition \ref{prop_imag2}, 
we will use \cite[Theorem 2.1]{MY_arch} instead of Lemma B.

\begin{proof}[Proof of Proposition \ref{prop_imag2}]
By using the facts mentioned in Remark \ref{rem_F-K},
we see that $X (k_\infty)$ is free of rank $r-1$ as a $\mathbb{Z}_2$-module.
We also note that $X (\mathbb{B}_\infty (\sqrt{\ell q}))$ is trivial 
by Iwasawa's theorem (\cite{Iwa56}, see also \cite[pp.438--439]{OT}).
Hence we see that 
$\rankztwo X (K_1) = \rankztwo X (k_\infty) = r-1$ 
by using the argument given in the proof of Proposition \ref{prop_imag}.

Moreover, we claim that $X (K_1)$ is $\mathbb{Z}_2$-torsion free.
We shall show this claim.
In the following, we denote $n$ by an arbitrary non-negative integer.
Note that $k_n (\sqrt{-q})$ is the $n$th layer of the 
cyclotomic $\mathbb{Z}_2$-extension $K_1 / k (\sqrt{-q})$.
Moreover, $k_n (\sqrt{-q})$ is a CM-field and $\mathbb{B}_n (\sqrt{\ell q})$ is its 
maximal real subfield.
Let $J$ be the generator of $\Gal (K_1 / \mathbb{B}_\infty (\sqrt{\ell q}))$.
Then, thanks to the result given by Atsuta \cite[Corollary 1.4]{Atsuta},
we can see that $X (K_1)/(1 +J) X (K_1)$ has no non-trivial $\mathbb{Z}_2$-torsion 
element.
We also see that $(1+J) X (k_n (\sqrt{-q}))$ is trivial since 
$X (\mathbb{B}_n (\sqrt{\ell q}))$ is trivial 
(this can be shown by considering the ideal class groups).
Hence $(1 +J) X (K_1)$ is trivial (cf. the argument given in the proof of 
\cite[Lemma 2.5]{Atsuta}).
The claim follows.

From this claim, we see that $X (K_1)$ is isomorphic to $X (k_\infty)$ via the 
restriction map.
Hence, by using \cite[Theorem 2.1]{MY_arch}, we see that 
$X (K_m) \cong X (k_\infty)$ for all $m \geq 2$, and then 
\[ X (K_\infty) \cong X (k_\infty) \cong \mathbb{Z}_2^{r-1} \]
(cf. the proof of \cite[Theorem 4.4]{MY_arch}).
The assertions of the proposition follows from this.
\end{proof}

\begin{rem}
In the above proof, we showed that $X (\mathbb{B}_\infty (\sqrt{-\ell}, \sqrt{-q})) 
\cong \mathbb{Z}_2^{r-1}$ as a $\mathbb{Z}_2$-module.
Similar results also can be found in \cite[Theorem 4.9]{MR}.
\end{rem}

\begin{acknowledgment}
The author express his gratitude to Takashi Fukuda. 
His talk and comment gave inspiration to start this study.
\end{acknowledgment}

\bigskip

\begin{flushleft}
Tsuyoshi Itoh \\
Division of Mathematics, 
Education Center,
Faculty of Innovative Management Science, \\
Chiba Institute of Technology, \\
2--1--1 Shibazono, Narashino, Chiba, 275--0023, Japan \\
e-mail : \texttt{tsuyoshi.itoh@it-chiba.ac.jp}

\end{flushleft}

\end{document}